\documentclass[11pt]{amsart}
\usepackage{mathrsfs,amsmath,amssymb,amsthm}
\usepackage{comment}
\usepackage{enumerate}
\newtheorem{theorem}{Theorem}

\newtheorem{Example}{Example}
\newtheorem{proposition}{Proposition}
\newtheorem{lemma}{Lemma}
\newtheorem{corollary}{Corollary}

\newcommand{\zz}{\mathbf z}
\newcommand{\xx}{\mathbf x}
\newcommand{\yy}{\mathbf y}
\newcommand{\ww}{\mathbf w}

\def\mapright#1{\smash{\mathop{\longrightarrow}\limits^{{#1}}}}
\def\mapdown#1{\Big\downarrow\rlap{$\vcenter{\hbox{$#1$}}$}}

\title{Join theorem for real analytic singularities}
\author{Kazumasa Inaba} 
\address{Faculty of Education, Iwate University, 18-33 Ueda 3-chome Morioka, Iwate 020-8550, Japan}
\email{inaba@iwate-u.ac.jp}

\begin{document}
\renewcommand{\thefootnote}{\fnsymbol{footnote}}
\footnote[0]{2010\textit{ Mathematics Subject Classification}.
Primary 32S55; Secondary: 32S40, 57Q45.}


\maketitle


\begin{abstract}
Let $f_{1} : (\Bbb{R}^{n}, \mathbf{0}_{n}) \rightarrow (\Bbb{R}^{p}, \mathbf{0}_{p})$ 
and $f_{2} : (\Bbb{R}^{m}, \mathbf{0}_{m}) \rightarrow (\Bbb{R}^{p}, \mathbf{0}_{p})$ 
be analytic germs of independent variables, where $n, m \geq p \geq 2$. 
In this paper, we assume that $f_{1}, f_{2}$ and $f = f_{1} + f_{2}$ satisfy $a_{f}$-condition. 
Then we show that 
the tubular Milnor fiber of $f$ is homotopy equivalent to 
the join of tubular Milnor fibers of $f_1$ and $f_2$. 
If $p = 2$, the monodromy of the tubular Milnor fibration of $f$ is equal to 
the join of the monodromies of the tubular Milnor fibrations of $f_1$ and $f_2$ up to homotopy.
\end{abstract}

\section{Introduction}
Let $f : (\Bbb{R}^{N}, \mathbf{0}_{N}) \rightarrow (\Bbb{R}^{p}, \mathbf{0}_{p})$
be an analytic germ, 
where $N \geq p \geq 2, \mathbf{0}_{N}$ 
and $\mathbf{0}_{p}$ are the origins of $\Bbb{R}^{N}$ and $\Bbb{R}^{p}$ respectively. 
Take a positive real number $\varepsilon_{0}$ sufficiently small if necessary. 
Assume that for any $0 < \varepsilon \leq \varepsilon_{0}$, 
there exists a positive real number $\delta$ such that 
$\delta \ll \varepsilon$ and 
\[
f : B^{N}_{\varepsilon} \cap f^{-1}(D^{p}_{\delta} \setminus \{\mathbf{0}_{p}\}) 
\rightarrow D^{p}_{\delta} \setminus \{\mathbf{0}_{p}\} 
\]
is a locally trivial fibration, 
where 
$B^{N}_{\varepsilon} = \{\xx \in \Bbb{R}^{N} \mid \| \xx \| \leq \varepsilon \}$ and  
$D^{p}_{\delta} = \{\ww \in \Bbb{R}^{p} \mid \| \ww \| \leq \delta \}$. 
The isomorphism class of the above fibration does not depend on the choice of $\varepsilon$ and $\delta$. 
This map is called the \textit{tubular Milnor fibration of $f$}. 
If $f_1$ and $f_2$ are holomorphic functions of independent variables, 
the following theorem is known. 

\begin{theorem}[Join theorem]
Let $f_{1} : (\Bbb{C}^{n}, O_{n}) \rightarrow (\Bbb{C}, 0)$ and 
$f_{2} : (\Bbb{C}^{m}, O_{m}) \rightarrow (\Bbb{C}, 0)$ be holomorphic functions of independent variables 
$\zz = (z_{1}, \dots, z_{n})$ and $\ww = (w_{1}, \dots, w_{m})$. 
Here $O_{N}$ is the origin of $\Bbb{C}^{N}$. 
Set $f(\zz, \ww) = f_{1}(\zz) + f_{2}(\ww)$. 
Then 
the Milnor fiber of $f$ is homotopy equivalent to the join of the Milnor fibers of $f_1$ and $f_2$ and 
the monodromy of $f$ is equal to the join of the monodromies of $f_1$ and $f_2$ up to homotopy. 
\end{theorem}
Join theorem is algebraically proved by M. Sebastiani and R. Thom for isolated singularities~\cite{ST}. 
M. Oka showed this for weighted homogeneous singularities \cite{O-1}. 
For general complex singularities, this is proved by K. Sakamoto \cite{S2}. 
In \cite{KN}, 
L. H. Kauffman and W. D. Neumann studied fiber structures and Seifert forms of 
links defined by tame isolated singularities of real analytic germs of independent variables. 
In this paper, we study Join theorem for more general real analytic singularities. 

To show the existence of Milnor fibrations for real analytic singularities, 
we consider stratifications of analytic sets. 
Let $f : (\Bbb{R}^{N}, \mathbf{0}_{N}) \rightarrow (\Bbb{R}^{p}, \mathbf{0}_{p})$ be a smooth map 
and $\mathcal{S}$ be a stratification of $B^{N}_{\varepsilon} \cap f^{-1}(0)$. 
The map $f$ satisfies \textit{$a_{f}$-condition}
if 
$B^{N}_{\varepsilon} \setminus f^{-1}(0)$ 
has no critical point and satisfies the following condition: 
For any sequence $p_{\nu} \in B^{N}_{\varepsilon} \setminus f^{-1}(0)$ 
such that 
\[
T_{p_{\nu}}f^{-1}(f(p_{\nu})) \rightarrow \tau, \ \ p_{\nu} \rightarrow p_{\infty} \in M, 
\]
where $M \in \mathcal{S}$, 
we have $T_{p_{\infty}}M \subset \tau$. 
A stratification $\mathcal{S}$ is called \textit{Whitney $(a)$-regular} 
if for any pair of strata $(S_{1}, S_{2})$ of $\mathcal{S}$ 
and any point $p \in S_{1} \cap \overline{S_{2}}$, 
$(S_{1}, S_{2})$ satisfies the following condition: 
For any sequence $q_{\nu} \in S_{2}$ satisfying 
\[
q_{\nu} \rightarrow p, \ \ T_{q_{\nu}}S_{2} \rightarrow T, 
\]
we have $T_{p}S_{1} \subset T$.

Let $f_{1} : (\Bbb{R}^{n}, \mathbf{0}_{n}) \rightarrow (\Bbb{R}^{p}, \mathbf{0}_{p})$ 
and $f_{2} : (\Bbb{R}^{m}, \mathbf{0}_{m}) \rightarrow (\Bbb{R}^{p}, \mathbf{0}_{p})$ 
be analytic germs, 
where $n, m \geq p \geq 2$. 
Set $V(f_{1}) = f_{1}^{-1}(0) \cap B^{n}_{\varepsilon}$ and 
$V(f_{2}) = f_{2}^{-1}(0) \cap B^{m}_{\varepsilon}$
for $0 < \varepsilon \ll 1$. 
We denote a stratification of 
$V(f_{1})$ (resp. $V(f_{2})$) 
by $\mathcal{S}_{1}$ (resp. $\mathcal{S}_{2}$). 
Assume that $f_{1}$ and $f_{2}$ satisfy the following conditions: 
\renewcommand{\theenumi}{\roman{enumi}}
\begin{enumerate}
\item
$\mathbf{0}_{p} \in \Bbb{R}^{p}$ is an isolated critical value of $f_{j}$ for $j = 1, 2$, 
\item
$f_{j}$ satisfies $a_{f}$-condition with respect to $\mathcal{S}_{j}$ for $j = 1, 2$. 
\end{enumerate}
\noindent
Since $V(f_{1})$ and $V(f_{2})$ are real analytic sets, we may assume that 
$\mathcal{S}_{1}$ and $\mathcal{S}_{2}$ are Whitney stratifications. See \cite{H}. 
Thus the stratifications $\mathcal{S}_{1}$ and $\mathcal{S}_{2}$ are Whitney $(a)$-regular. 
We take $\varepsilon$ sufficiently small if necessary. 
Then the sphere $\partial B^{n}_{\varepsilon}$ (resp. $\partial B^{m}_{\varepsilon}$) 
intersects $M_{1}$ (resp. $M_{2}$)
transversely for any $M_{1} \in \mathcal{S}_{1}$ and $M_{2} \in \mathcal{S}_{2}$. 
See the proof of \cite[Lemma 3.2]{BV}. 

Put $U_{1} = \{ \xx \in B^{n}_{\varepsilon} \mid \| f_{1}(\xx)\| \leq \delta \}$ 
and $U_{2} = \{ \yy \in B^{m}_{\varepsilon} \mid \| f_{2}(\yy)\| \leq \delta \}$, 
where $0 < \delta \ll \varepsilon$. 
By the above conditions and the Ehresmann fibration theorem \cite{W}, we may assume that 
\[
f_{j} : U_{j} \setminus V(f_{j}) \rightarrow D^{p}_{\delta} \setminus \{0\}
\]
is a locally trivial fibration for $j = 1, 2$. 

Let $f : (\Bbb{R}^{n} \times \Bbb{R}^{m}, \mathbf{0}_{n+m}) \rightarrow (\Bbb{R}^{p}, \mathbf{0}_{p})$ 
be the analytic germ 
defined by $f = f_{1} + f_{2}$. Put $V(f) = f^{-1}(0) \cap (U_{1} \times U_{2})$. 
By \cite[Proposition 5.2]{ACT}, $f$ also satisfies the conditions (i) and (ii). See Section $2.1$. 
The main theorem of this paper is the following. 

\begin{theorem}
Let $f_{1} : (\Bbb{R}^{n}, \mathbf{0}_{n}) \rightarrow (\Bbb{R}^{p}, \mathbf{0}_{p})$ 
and $f_{2} : (\Bbb{R}^{m}, \mathbf{0}_{m}) \rightarrow (\Bbb{R}^{p}, \mathbf{0}_{p})$ 
be analytic germs of independent variables, where $n, m \geq p \geq 2$. 
Assume that $f_1$ and $f_2$ satisfy the conditions (i) and (ii). 
Set $f = f_{1} + f_{2}$.  
Then the fiber of the tubular Milnor fibration of $f$ is homotopy equivalent to 
the join of the fibers of the tubular Milnor fibrations of $f_1$ and $f_2$. 

Moreover, if $p = 2$, the monodromy of the tubular Milnor fibration of $f$ is equal to 
the join of the monodromies of $f_1$ and $f_2$ up to homotopy. 
\end{theorem}

Moreover, we assume that $f_{1}, f_{2}$ and $f$ satisfy
the following condition: 
\begin{enumerate}
\setcounter{enumi}{2}
\item
there exists a positive real number $r'$ such that 
\[
P/\lvert P\rvert : \partial B^{N}_{r} \setminus K_{P} \rightarrow S^{p-1}
\]
is a locally trivial fibration and this fibration is isomorphic to the tubular Milnor fibration of $P$, 
where $K_P = \partial B^{N}_{r} \cap P^{-1}(0)$ and $0 < r \leq r'$ for 
$(P, N) = (f_{1}, n), (f_{2}, m), (f, n+m)$. 
\end{enumerate}
The fibration in (iii) is called \textit{the spherical Milnor fibration of $P$}. 
By using Theorem $2$ and the condition (iii), we have 
\begin{corollary}
Let $f_{1} : (\Bbb{R}^{n}, \mathbf{0}_{n}) \rightarrow (\Bbb{R}^{p}, \mathbf{0}_{p})$ 
and $f_{2} : (\Bbb{R}^{m}, \mathbf{0}_{m}) \rightarrow (\Bbb{R}^{p}, \mathbf{0}_{p})$ 
be analytic germs in Theorem $2$. 
Assume that $f_{1}, f_{2}$ and 
$f = f_{1} + f_{2}$ satisfy the condition (iii). 
Then the fiber of the spherical Milnor fibration of $f$ is 
homotopy equivalent to the join of 
the fibers of the spherical Milnor fibrations of $f_{1}$ and $f_{2}$. 
\end{corollary}

If $p$ is equal to $2$, analytic germs which satisfy the above conditions were studied by Oka. 
Let $(\rho_{1}, \rho_{2}) : (\Bbb{R}^{2n}, \mathbf{0}_{2n}) \rightarrow (\Bbb{R}^{2}, \mathbf{0}_{2})$ 
be an analytic map germ 
with real $2n$-variables $x_{1}, \dots, x_{n}$ and $y_{1}, \dots, y_{n}$. 
Then $(\rho_{1}, \rho_{2})$ is represented by 
a function of variables $\zz = (z_{1}, \dots, z_{n})$ and $\bar{\zz} = (\bar{z}_{1}, \dots, \bar{z}_{n})$ as 
\[ 
P(\zz, \bar{\zz}) := 
\rho_{1}\Bigl(\frac{\zz + \bar{\zz}}{2}, \frac{\zz - \bar{\zz}}{2\sqrt{-1}}\Bigr) 
+\sqrt{-1}\rho_{2}\Bigl(\frac{\zz + \bar{\zz}}{2}, \frac{\zz - \bar{\zz}}{2\sqrt{-1}}\Bigr). 
\]
Here any complex variable $z_{j}$ of $\Bbb{C}^{n}$ is represented by 
$x_{j} + \sqrt{-1}y_{j}$ and $\bar{z}_{j}$ is the complex conjugate of $z_j$ for $j = 1, \dots, n$. 
Then a map $P: (\Bbb{C}^{n}, O_{n}) \rightarrow (\Bbb{C}, 0)$ is called a \textit{mixed function map}. 
For mixed weighted homogeneous singularities, 
Join theorem is proved 
by J. L. Cisneros-Molina~\cite{C}. 
Oka introduced the notion of Newton boundaries of mixed functions and the concept of strongly non-degeneracy. 
If $P$ is a convenient strongly non-degenerate mixed function or 
a strongly non-degenerate mixed function which 
is locally tame along vanishing coordinate subspaces, then $P$ satisfies the conditions (i), (ii) and (iii). 
See \cite{O1, O2, EO}.

We study the topology of Milnor fibrations of join type. 
If a mixed function $P$ satisfies the condition (iii) and the origin is an isolated singularity of $P$, 
the Seifert form is determined by the spherical Milnor fibration of $P$. 
Note that the Seifert form is a topological invariant of fibrations. 
Then we calculate Seifert forms defined by joins of Milnor fibrations of $1$-variable mixed functions. 
This is a generalization of \cite[Corollary 3]{S1}.

We also study homotopy types of fibered links defined by isolated singularities of join type. 
In \cite{NR1, NR2, NR3}, W. Neumann and L. Rudolph defined 
the enhanced Milnor number and the enhancement to the Milnor number 
of a fibered link. 
These are invariants of homotopy types of fibered links in $S^{2k+1}$. 
If $k=1$, for any $d \in \Bbb{Z}$, there exists a mixed polynomial $P$ such that 
the enhancement to the Milnor number of $K_P$ is equal to $d$ \cite{In1}. 
If $k$ is greater than $1$, 
the enhanced Milnor number is represented by $((-1)^{k+1}\ell, r)$, 
where $\ell \in \Bbb{N}$ and $r \in \{0, 1\}$. 
Note that there exists a complex polynomial $Q$ such that 
the enhanced Milnor number determined by the Milnor fibration of $Q$ is equal to 
$((-1)^{k+1}\ell, 0)$ for $\ell \in \Bbb{N}$ and $k \geq 2$. 
We show that there exists a mixed polynomial of join type such that 
the enhanced Milnor number of a link defined by a mixed polynomial 
is equal to 
$((-1)^{k+1}\ell, 1)$ for $\ell \in \Bbb{N}$ and $k \geq 2$.

This paper is organized as follows. In Section $2$ 
we give some Join type statements, the definition of zeta functions of monodromies and 
strongly non-degenerate mixed functions. 
In Section $3$ we prove Theorem $1$ and Corollary $1$. 
In Section $4$ we consider Join theorem of 
Seifert forms of links defined by $1$-variable mixed polynomials. 
In Section $5$ we study homotopy types of Milnor fibrations defined by 
mixed polynomial of Join type.

The author would like to thank  
Masaharu Ishikawa, Mutsuo Oka and Mihai Tib\u{a}r for precious comments and fruitful suggestions. 

\section{Preliminaries}
\subsection{Join type statements}
Let $f : (\Bbb{R}^{N}, \mathbf{0}_{N}) \rightarrow (\Bbb{R}^{p}, \mathbf{0}_{p})$ 
be an analytic germ which satisfies the conditions (i) and (ii). 
By using the same argument in \cite[Proposition 11]{O2}, we can show the following lemma. 
\begin{lemma}
Assume that an analytic germ $f : (\Bbb{R}^{N}, \mathbf{0}_{N}) \rightarrow (\Bbb{R}^{p}, \mathbf{0}_{p})$ 
satisfies the conditions (i) and (ii). 
There exists a sufficiently small positive real number $r_{0}$ which satisfies the following: 
For any positive real number $r_{1}$ which satisfies $r_{1} \leq r_{0}$, 
there exists a positive real number $\tilde{\delta}$ such that $f^{-1}(\eta)$ intersects transversely with the sphere 
$S^{N-1}_{r}$ for $r_{1} \leq r \leq r_{0}$ and $0 < \| \eta\| \leq \tilde{\delta}$. 
\end{lemma}

Let $f_{1} : (\Bbb{R}^{n}, \mathbf{0}_{n}) \rightarrow (\Bbb{R}^{p}, \mathbf{0}_{p})$ 
and $f_{2} : (\Bbb{R}^{m}, \mathbf{0}_{m}) \rightarrow (\Bbb{R}^{p}, \mathbf{0}_{p})$ 
be analytic germs which satisfy the conditions (i) and (ii), where $n, m \geq p \geq 2$. 
Let $f : (\Bbb{R}^{n} \times \Bbb{R}^{m}, \mathbf{0}_{n+m}) \rightarrow (\Bbb{R}^{p}, \mathbf{0}_{p})$ 
be the analytic germ 
defined by $f = f_{1} + f_{2}$. Put $V(f) = f^{-1}(0) \cap (U_{1} \times U_{2})$. 
We take the stratification $\mathcal{S}$ of $V(f)$ as follows: 
\[
\mathcal{S} : 
(\mathcal{S}_{1} \times \mathcal{S}_{2}) \sqcup 
\bigl( V(f) \setminus (V(f_{1}) \times V(f_{2})) \bigr), 
\]
where $\mathcal{S}_{j}$ is a stratification of $V(f_{j})$ in Section~$1$ for $j = 1, 2$. 
By \cite{GWPL}, we may assume that $\mathcal{S}_{1} \times \mathcal{S}_{2}$ is Whitney $(a)$-regular. 
By using the stratification $\mathcal{S}$ of $V(f)$, 
R. N. Ara\'{u}jo dos Santos, Y. Chen and M. Tib\u{a}r showed the following lemma.  

\begin{lemma}[{\cite[Proposition 5.2]{ACT}}]
Let $f_{1} : (\Bbb{R}^{n}, \mathbf{0}_{n}) \rightarrow (\Bbb{R}^{p}, \mathbf{0}_{p})$ 
and $f_{2} : (\Bbb{R}^{m}, \mathbf{0}_{m}) \rightarrow (\Bbb{R}^{p}, \mathbf{0}_{p})$ 
be analytic germs which satisfy the conditions (i) and (ii), where $n, m \geq p \geq 2$. 
Then the analytic germ 
$f = f_{1} + f_{2} : (\Bbb{R}^{n} \times \Bbb{R}^{m}, \mathbf{0}_{n+m}) \rightarrow (\Bbb{R}^{p}, \mathbf{0}_{p})$ 
also satisfies the conditions (i) and (ii). 
\end{lemma}

By Lemma $1$ and Lemma $2$, we have 
\begin{corollary}
There exists a positive real number $\varepsilon'_{0}$ such that
the restricted map $f|_{B^{n+m}_{\varepsilon'}}: B^{n+m}_{\varepsilon'} \rightarrow \Bbb{R}^{p}$ also satisfies 
the conditions (i) and (ii) for $0 < \varepsilon' \leq \varepsilon'_{0}$. 
Moreover, there exists a positive real number $\delta'$ such that 
\[
f : B^{n+m}_{\varepsilon'} \cap f^{-1}(D^{p}_{\delta'} \setminus \{ \mathbf{0}_{p}\}) 
\rightarrow D^{p}_{\delta'} \setminus \{ \mathbf{0}_{p}\}
\] 
is a locally trivial fibration. 
\end{corollary}

\subsection{Divisors and Zeta functions of monodromies}
Take $1$-variable polynomials $q_{1}(t)$ and $q_{2}(t)$ with $q_{1}(0) = q_{2}(0) = 0$. 
Set $q_{1}(t) = \alpha_{0}\prod_{j=1}^{k}(t - \alpha_{i})$ and 
$q_{2}(t) = \beta_{0}\prod_{j=1}^{\ell}(t - \beta_{j})$, 
where $\alpha_{i}, \beta_{j} \in \Bbb{C}^{*} := \Bbb{C} \setminus \{0\}$
for $i = 0, \dots, k$ and $j = 0, \dots, \ell$. 
Then we define the \textit{divisor of $q_{1}(t)/q_{2}(t)$} by 
\[
\biggl(\frac{q_{1}(t)}{q_{2}(t)}\biggr) = \sum_{i=1}^{k}\langle \alpha_{i}\rangle - 
\sum_{j=1}^{\ell}\langle \beta_{j}\rangle
\in \Bbb{Z}(\Bbb{C}^{*}).
\]

Let $F$ be the fiber of the spherical Milnor fibration of 
$P : (\Bbb{R}^{2n}, \mathbf{0}_{2n}) \rightarrow (\Bbb{R}^{2}, \mathbf{0}_{2})$ 
and 
$h : F \rightarrow F$ be the monodromy of this fibration. 
Set $P_{j}(t) = \det (\text{Id} - th_{*, j})$, where 
$h_{*, j} : H_{j}(F, \Bbb{Q}) \rightarrow H_{j}(F, \Bbb{Q})$ is an isomorphism induced by $h$. 
Then the \textit{zeta function $\zeta(t)$ of the monodromy} is defined by 
\[
\zeta(t) = \prod_{j=0}^{2n-2}P_{j}(t)^{(-1)^{j+1}}.
\]
See \cite[Chapter I]{O0}. Assume that $P$ satisfies the following properties: 
\renewcommand{\theenumi}{\alph{enumi}}
\begin{enumerate}
\item
$\mathbf{0}_{2n}$ is an isolated singularity of $P$,
\item
$F$ has a homotopy type of a finite CW-complex of dimension $\leq n-1$,
\item
$F$ is $(n-2)$-connected.
\end{enumerate}
Then the zeta function $\zeta(t)$ is equal to $P_{n-1}(t)^{(-1)^{n}}/(t-1)$ and 
the \textit{reduced zeta function} is defined by $\tilde{\zeta}(t) = (t - 1)\zeta(t)$.

\subsection{Strongly non-degenerate mixed functions}
In this subsection, we introduce a class of mixed functions which admit 
tubular Milnor fibrations and spherical Milnor fibrations 
given by Oka in \cite{O1}. 
Let $P(\zz, \bar{\zz})$ be a mixed function, i.e., 
$P(\zz, \bar{\zz})$ is a function 
expanded in a convergent power series 
of variables $\zz = (z_1, \dots, z_n)$ and $\bar{\zz} = (\bar{z}_1, \dots, \bar{z}_n)$ 
\[
P(\zz, \bar{\zz}) := \sum_{\nu, \mu} c_{\nu, \mu}\zz^{\nu}\bar{\zz}^{\mu}, 
\]
where $\zz^{\nu} = z^{\nu_1}_1 \cdots z^{\nu_n}_n$ for $\nu = (\nu_1, \dots, \nu_n)$ 
(respectively $\bar{\zz}^{\mu} = \bar{z}_{1}^{\mu_1} \cdots \bar{z}_{n}^{\mu_n}$ for $\mu = (\mu_1, \dots, \mu_n))$. 
The \textit{Newton polygon} $\Gamma_{+}(P;\zz.\bar{\zz})$ is defined by the convex hull of 
\[
\bigcup_{(\nu, \mu)}\{(\nu + \mu)+\Bbb{R}^n_{+} \ | 
\ c_{\nu, \mu}\neq0\}, 
\]
where $\nu + \mu$ is the sum of the multi-indices of $\zz^{\nu}\bar{\zz}^{\mu}$, 
i.e., $\nu+\mu = (\nu_{1}+\mu_{1}, \dots, \nu_{n}+\mu_{n})$. 
\textit{The Newton boundary} $\Gamma(P;\zz,\bar{\zz})$ is the union of compact faces  of $\Gamma_+(P;\zz,\bar{\zz})$. 
The \textit{strongly non-degeneracy} is defined from the Newton boundary as follows:  
let $\Delta_1,$ $\dots$ ,$\Delta_m$ be the faces of $\Gamma(P;\zz,\bar{\zz})$. 
For each face $\Delta_{k}$,  
the \textit{face function} 
$P_{\Delta_{k}}(\zz,\bar{\zz})$ is defined by $P_{\Delta_{k}}(\zz,\bar{\zz}) := \sum_{(\nu + \mu) \in \Delta_{k}}
c_{\nu, \mu}\zz^{\nu}\bar{\zz}^{\mu}$. 
If $P_{\Delta_{k}}(\zz,\bar{\zz}) : \Bbb{C}^{\ast n} \rightarrow \Bbb{C}$ 
has no critical point and 
$P_{\Delta_{k}}$ is surjective for $\dim \Delta_{k} \geq 1$, 
we say that $P(\zz, \bar{\zz})$ is \textit{strongly non-degenerate} for $\Delta_{k}$, where 
$\Bbb{C}^{\ast n} = \{\zz = (z_{1}, \dots, z_{n}) \mid z_{j} \neq 0, j = 1, \dots, n \}$. 
If $P(\zz, \bar{\zz})$ is strongly non-degenerate for any $\Delta_{k}$ for $k = 1,\dots,m$,
we say that $P(\zz, \bar{\zz})$ is 
{\textit{strongly non-degenerate}}. 
If $P( (0,\dots,0, z_{j},0,\dots,0), (0,\dots,0, \bar{z}_{j},0,\dots,0)) \not\equiv 0$ 
for each $j= 1,\dots,n$, 
then we say that $P(\zz, \bar{\zz})$ is \textit{convenient}. 
Oka showed that 
a convenient strongly non-degenerate mixed function $P(\zz, \bar{\zz})$ has the Milnor fibration.  

\begin{theorem}[\cite{O1, O2, EO}] 
Let $P(\zz, \bar{\zz}) : (\Bbb{C}^{n}, O_{n}) \rightarrow (\Bbb{C}, 0)$ be a convenient 
strongly non-degenerate mixed function. 
Then 
$O_{n}$ is an isolated singularity of $P$ and 
$P$ satisfies the conditions (i), (ii) and (iii). 
\end{theorem} 

Let $f_t$ be an analytic family of convenient strongly non-degenerate mixed polynomials 
such that the Newton boundary of $f_t$ is constant for $0 \leq t \leq 1$. 
C. Eyral and M. Oka showed that the topological type of $(V(f_{t}), O_{n})$ is constant for any $t$ 
and their tubular Milnor fibrations are equivalent \cite{EO}.

\section{Proof of Theorem $2$}
Assume that $f_{1}$ and $f_{2}$ satisfy the conditions (i) and (ii) in Section $1$. 
Then the proof of Theorem~$2$ is analogous to the holomorphic case \cite{S2}. 
Set $X_{\mathbf{t}} = f_{1}^{-1}(\mathbf{t}) \cap U_{1}, 
Y_{\mathbf{t}} = f_{2}^{-1}(\mathbf{t}) \cap U_{2}$ and 
$Z_{\mathbf{t}} = f^{-1}(\mathbf{t}) \cap (U_{1} \times U_{2})$. 
Take a positive real number $\delta'$ as in Corollary $1$. 
We fix a point $\mathbf{t} \in \Bbb{R}^{p}$ with 
$0 < \| \mathbf{t} \| \ll \delta'$ and define the map 
\begin{center}
\text{$F_{1} : Z_{\mathbf{t}} \rightarrow A$ as $(\xx,\yy) \mapsto f_{1}(\xx)$}, 
\end{center}
where $A = \{ \ww \in \Bbb{R}^{p} \mid \| \mathbf{t} - \ww \| \leq \delta' \}$. 

\begin{lemma}
The restriction map 
$F_{1} : Z_{\mathbf{t}} \setminus F_{1}^{-1}(\{ \mathbf{0}_{p}, \mathbf{t} \}) \rightarrow 
A \setminus \{ \mathbf{0}_{p}, \mathbf{t} \}$ 
is a locally trivial fibration. 
\end{lemma}
\begin{proof}
From the tubular Milnor fibrations of $f_1$ and $f_2$, for each $\ww \in A \setminus \{ \mathbf{0}_{p}, \mathbf{t} \}$, 
we may find a neighborhood $V_{\ww} \subset A \setminus \{ \mathbf{0}_{p}, \mathbf{t} \}$ of $\ww$ such that 
there exist local trivializations 
\[
\phi_{1} : V_{\ww} \times X_{\ww} \overset{\cong}{\longrightarrow} f_{1}^{-1}(V_{\ww}) \cap U_{1}, \ \
\phi_{2} : V_{\mathbf{t} - \ww} \times Y_{\mathbf{t} - \ww} \overset{\cong}{\longrightarrow} 
f_{2}^{-1}(V_{\mathbf{t} - \ww}) \cap U_{2},
\]
where $V_{\mathbf{t} - \ww} = \{ \mathbf{t} - \ww \mid \ww \in V_{\ww} \} \subset A \setminus \{ \mathbf{0}_{p}, \mathbf{t} \}$. 
We define the map on $V_{\ww} \times F_{1}^{-1}(\ww) = 
V_{\ww} \times (X_{\ww} \times Y_{\mathbf{t} - \ww})$ as follows: 
\[
\psi : V_{\ww} \times (X_{\ww} \times Y_{\mathbf{t} - \ww}) \rightarrow F_{1}^{-1}(V_{\ww}), \ \ 
(\ww', \xx, \yy) \mapsto (\phi_{1}(\ww', \xx), \phi_{2}(\ww', \yy)). 
\]
Since $\phi_{1}$ and $\phi_{2}$ are local trivializations, $\psi$ is a continuous map. 
For any $(\xx', \yy') \in F_{1}^{-1}(V_{\ww})$, we put 
$(\ww', \xx) = \phi_{1}^{-1}(\xx')$ and $(\mathbf{t} - \ww', \yy) = \phi_{2}^{-1}(\yy')$. 
Then $\psi^{-1}(\xx', \yy')$ is equal to $(\ww', \xx, \yy)$. 
The map $\psi^{-1}$ is a continuous map. 
Thus $\psi$ is a homeomorphism. This shows the local triviality of~$F_{1}$. 
\end{proof}

\begin{lemma}
Let $J$ be the line segment with endpoints $\mathbf{0}_{p}$ and $\mathbf{t}$. 
The inclusion $F_{1}^{-1}(J) \hookrightarrow Z_{\mathbf{t}}$ is a homotopy equivalence. 
\end{lemma}
\begin{proof}
Since $Z_{\mathbf{t}}$ is semi-analytic, 
there is a triangulation of $Z_{\mathbf{t}}$ such that 
$F_{1}^{-1}(J)$ is a subcomplex~\cite{Lo}. 
Since $Z_{\mathbf{t}}$ is compact, by using the local triviality of $F_1$ and the partition of unity, 
$Z_{\mathbf{t}}$ is deformed into a regular neighborhood of $F_{1}^{-1}(J)$. 
Thus $F_{1}^{-1}(J)$ and $Z_{\mathbf{t}}$ are homotopy equivalent. 
See \cite[Chapter 3]{RS}. 
\end{proof}


Let $\pi : U_{1} \times U_{2} \rightarrow (U_{1}/V(f_{1})) \times (U_{2}/V(f_{2}))$ be the identification map. 
\begin{lemma}
The identification map $\pi : F_{1}^{-1}(J) \rightarrow \pi(F_{1}^{-1}(J))$ is a homotopy equivalence. 
\end{lemma}
\begin{proof}
The semi-analytic set $V(f_{j})$ has a conic structure for $j = 1, 2$ \cite{BV}, i.e., 
\begin{equation*}
\begin{split}
V(f_{1})& \cong \text{Cone}(V(f_{1}) \cap S^{n-1}_{\varepsilon}) = 
([0, 1] \times (V(f_{1}) \cap S^{n-1}_{\varepsilon}))/(\{0\} \times (V(f_{1}) \cap S^{n-1}_{\varepsilon})), \\ 
V(f_{2})& \cong \text{Cone}(V(f_{2}) \cap S^{m-1}_{\varepsilon}) = 
([0, 1] \times (V(f_{2}) \cap S^{m-1}_{\varepsilon}))/(\{0\} \times (V(f_{2}) \cap S^{m-1}_{\varepsilon})). 
\end{split}
\end{equation*}
So $V(f_{1})$ and $V(f_{2})$ contract to the origins of $\Bbb{R}^{n}$ and $\Bbb{R}^{m}$ respectively. 
We can construct deformation retractions from 
$F_{1}^{-1}(\mathbf{0}_{p}) = V(f_{1}) \times Y_{\mathbf{t}}$ to 
$\{ \mathbf{0}_{n}\} \times Y_{\mathbf{t}}$ and from 
$F_{1}^{-1}(\mathbf{t}) = X_{\mathbf{t}} \times V(f_{2})$ to 
$X_{\mathbf{t}} \times \{ \mathbf{0}_{m}\}$. 
By applying a triangulation of $F_{1}^{-1}(J)$ and 
using the homotopy extension property of a polyhedral pair \cite{Sp}, 
the above homotopies can extend to a homotopy 
$H_{s} : F_{1}^{-1}(J) \rightarrow F_{1}^{-1}(J)$ so that 
\[
H_{0} = \text{id}_{F_{1}^{-1}(J)}, \ \ 
H_{1}(F_{1}^{-1}(\mathbf{0}_{p}) \cup F_{1}^{-1}(\mathbf{t})) = 
\{ \mathbf{0}_{n}\} \times Y_{\mathbf{t}} \cup X_{\mathbf{t}} \times \{ \mathbf{0}_{m}\}, 
\]
where $0 \leq s \leq 1$. 
Let $\tilde{H}_{s} : \pi(F_{1}^{-1}(J)) \rightarrow \pi(F_{1}^{-1}(J))$ be the homotopy 
which satisfies $\pi(H_{s}(\xx, \yy)) = \tilde{H}_{s}(\pi(\xx, \yy))$, 
where $(\xx, \yy) \in F_{1}^{-1}(J)$ and $0 \leq s \leq 1$. 
Note that 
$\pi(F_{1}^{-1}(J)) \setminus 
(\{ \mathbf{0}_{n}\} \times Y_{\mathbf{t}} \cup X_{\mathbf{t}} \times \{ \mathbf{0}_{m}\}) = 
F_{1}^{-1}(J) \setminus (F_{1}^{-1}(\mathbf{0}_{p}) \cup F_{1}^{-1}(\mathbf{t}))$. 
The map $\varphi : \pi(F_{1}^{-1}(J)) \rightarrow F_{1}^{-1}(J)$ is defined by 
\begin{equation*}
\begin{split}
&\varphi \mid_{\pi(F_{1}^{-1}(J)) \setminus (\{ \mathbf{0}_{n}\} \times Y_{\mathbf{t}} \cup X_{\mathbf{t}} \times \{ \mathbf{0}_{m}\})} = 
H_{1} \mid_{F_{1}^{-1}(J) \setminus (F_{1}^{-1}(\mathbf{0}_{p}) \cup F_{1}^{-1}(\mathbf{t}))}, \\
&\varphi(\{ \mathbf{0}_{n}\} \times Y_{\mathbf{t}}) = \{ \mathbf{0}_{n}\} \times Y_{\mathbf{t}}, \ \
\varphi(X_{\mathbf{t}} \times \{ \mathbf{0}_{m}\}) = X_{\mathbf{t}} \times \{ \mathbf{0}_{m}\}. 
\end{split}
\end{equation*}
Then $\varphi$ is continuous and $H_{1} = \varphi \circ \pi$. 
By the definition of $\tilde{H}_{s}$, $\pi \circ \varphi = \tilde{H}_{1}$. 
Thus the identification map $\pi$ is a homotopy equivalence. 
\end{proof}

\begin{lemma}
Let $X_{\mathbf{t}} * Y_{\mathbf{t}}$ be the join of $X_{\mathbf{t}}$ and $Y_{\mathbf{t}}$. 
Then $X_{\mathbf{t}} * Y_{\mathbf{t}}$ is homeomorphic to $\pi(F_{1}^{-1}(J))$.  
\end{lemma}
\begin{proof}
Put $I = [0, 1]$. 
By the local trivialities of the tubular Milnor fibrations of $f_{1}$ and $f_{2}$, 
there exist homeomorphisms 
\[
\tilde{\phi}_{1} : (I \setminus \{0\}) \times X_{\mathbf{t}} \rightarrow f_{1}^{-1}(J \setminus \{ \mathbf{0}_{p} \}) \cap U_{1}, \ \ 
\tilde{\phi}_{2} : (I \setminus \{0\}) \times Y_{\mathbf{t}} \rightarrow f_{2}^{-1}(J \setminus \{ \mathbf{0}_{p} \}) \cap U_{2}
\] 
such that $f_{1}(\tilde{\phi}_{1}(s, \xx)) = f_{2}(\tilde{\phi}_{2}(s, \yy)) = s\mathbf{t}$ for $0 < s \leq 1$. 
We define the map 
\begin{center}
$\Phi : X_{\mathbf{t}} \times I \times Y_{\mathbf{t}} \rightarrow \pi(F_{1}^{-1}(J))$ as 
$(\xx, s, \yy) \mapsto \pi(\tilde{\phi}_{1}(s, \xx), \tilde{\phi}_{2}(1 - s, \yy))$, 
\end{center}
where $\tilde{\phi}_{1}(0, \xx) = \mathbf{0}_{n}$ and  $\tilde{\phi}_{2}(0, \yy) = \mathbf{0}_{m}$. 
Since $V(f_{1})$ and $V(f_{2})$ have conic structures, 
$\Phi$ is a continuous map. 
Let $\Psi : X_{\mathbf{t}} * Y_{\mathbf{t}} \rightarrow \pi(F_{1}^{-1}(J))$ be the map 
defined by $\Psi([\xx, s, \yy]) = \Phi(\xx, s, \yy)$, where 
$[\xx, s, \yy]$ is the equivalence class of $(\xx, s, \yy)$. 
By the definition of $\Phi$ and conic structures of $V(f_{1})$ and $V(f_{2})$, 
$\Psi$ is a continuous and bijective map. 
Thus $\Psi$ is a homeomorphism. 
\end{proof}


\begin{lemma}
The fiber $Z_{\mathbf{t}}$ is homotopy equivalent to $f^{-1}(\mathbf{t}) \cap D^{n+m}_{\varepsilon'}$, 
where $0 < \varepsilon' \ll 1$. 
\end{lemma}
\begin{proof}
By Lemma $1$, we can choose positive real numbers 
$\varepsilon_{1}$ and $\varepsilon_{2}$ such that $\varepsilon_{1} < \varepsilon_{2} \ll 1$ and the inclusion 
\[
f^{-1}(\mathbf{t}) \cap D^{n+m}_{\varepsilon_{1}} \hookrightarrow f^{-1}(\mathbf{t}) \cap D^{n+m}_{\varepsilon_{2}} 
\]
is a homotopy equivalence. 
Since $U_{1} \subset B^{n}_{\varepsilon}, U_{2} \subset B^{m}_{\varepsilon}$ and 
$\varepsilon$ is sufficiently small, 
we can also choose $\varepsilon_{1}$ and $\varepsilon_{2}$ which satisfy 
\[
D^{n+m}_{\varepsilon_{1}} \subset (U_{1} \times U_{2}) \subset 
D^{n+m}_{\varepsilon_{2}}. 
\]
By $a_{f}$-condition of $f_{j}$ and Lemma $1$, 
there exist neighborhoods $U'_{1}$ of $\mathbf{0}_{n}$ and $U'_{2}$ of $\mathbf{0}_{m}$ such that 
the inclusion 
\[
\pi(F_{1}^{-1}(J) \cap U'_{1} \times U'_{2}) \hookrightarrow \pi(F_{1}^{-1}(J))
\] 
is a homotopy equivalence and 
\[
(U'_{1} \times U'_{2}) \subset D^{n+m}_{\varepsilon_{1}} \subset (U_{1} \times U_{2}) \subset 
D^{n+m}_{\varepsilon_{2}}. 
\]
By using Lemma $4$, Lemma $5$ and the above homotopy, we have 
\begin{equation*}
\begin{split}
f^{-1}(\mathbf{t}) \cap (U'_{1} \times U'_{2}) \simeq Z_{\mathbf{t}} \cap (U'_{1} \times U'_{2}) 
&\simeq F_{1}^{-1}(J) \cap (U'_{1} \times U'_{2}) \\
&\simeq \pi(F_{1}^{-1}(J) \cap U'_{1} \times U'_{2}) \simeq \pi(F_{1}^{-1}(J)) \simeq Z_{\mathbf{t}}. 
\end{split}
\end{equation*}
Here $\simeq$ denotes a homotopy equivalence. 
Since $f^{-1}(\mathbf{t}) \cap (U'_{1} \times U'_{2})$ is homotopy equivalent to $Z_{\mathbf{t}}$ and 
$(U'_{1} \times U'_{2}) \subset D^{n+m}_{\varepsilon_{1}} \subset (U_{1} \times U_{2})$, 
this homotopy equivalence induces the following isomorphism of homotopy groups: 
\[
\pi_{u}(f^{-1}(\mathbf{t}) \cap (U'_{1} \times U'_{2})) \cong 
\pi_{u}(f^{-1}(\mathbf{t}) \cap D^{n+m}_{\varepsilon_{1}}) \cong \pi_{u}(Z_{\mathbf{t}}), 
\]
where $u \geq 0$. 
So the inclusion map $f^{-1}(\mathbf{t}) \cap D^{n+m}_{\varepsilon_{1}} \hookrightarrow Z_{\mathbf{t}}$ 
is a weak homotopy equivalence. 
Since $f^{-1}(\mathbf{t}) \cap D^{n+m}_{\varepsilon_{1}}$ and $Z_{\mathbf{t}}$ are CW-complexes, 
$Z_{\mathbf{t}}$ is homotopy equivalent to $f^{-1}(\mathbf{t}) \cap D^{n+m}_{\varepsilon_{1}}$ \cite{Sp}. 
Since $\varepsilon'$ and $\| t\|$ are sufficiently small, by Lemma $1$, 
$Z_{\mathbf{t}}$ is homotopy equivalent to $f^{-1}(\mathbf{t}) \cap D^{n+m}_{\varepsilon'}$. 
\end{proof}

\begin{proof}[Proof of Theorem 2]
By using Lemma $4$, Lemma $5$ and Lemma $6$, we can show that $X_{\mathbf{t}} * Y_{\mathbf{t}}$ 
is homotopy equivalent to $Z_{\mathbf{t}}$. 
By Lemma $7$, the fiber of the tubular Milnor fibration of $f$ is homotopy equivalent to 
$X_{\mathbf{t}} * Y_{\mathbf{t}}$. 

If $p =2$,  set 
\[
E = \{ (\xx, \yy) \in U_{1} \times U_{2} \mid 0 < \| f(\xx, \yy)\| \leq \rho \},
\]
where $0 < \rho \ll 1$. 
Then the map $\tilde{f} : \pi(E) \rightarrow  D^{2}_{\rho} \setminus \{ \mathbf{0}_{2}\}$ is defined by 
$\tilde{f}(\pi(\xx, \yy)) = f(\xx, \yy)$. 
By the local trivialities of $f_{1}$ and $f_{2}$, there are continuous 
one-parameter families of homeomorphisms
\[
\alpha_{\theta} : U_{1} \setminus V(f_{1}) \rightarrow U_{1} \setminus V(f_{1}), \ \ 
\beta_{\theta} : U_{2} \setminus V(f_{2}) \rightarrow U_{2} \setminus V(f_{2})
\]
such that $f_{1}(\alpha_{\theta}(\xx)) = e^{i\theta}f_{1}(\xx)$ and 
$f_{2}(\beta_{\theta}(\yy)) = e^{i\theta}f_{2}(\yy)$, where $\theta \in [0, 2\pi]$. 
Then we define the map
$\gamma_{\theta} : \pi(E) \rightarrow \pi(E)$ as follows: 
\[ 
\gamma_{\theta}(\pi(\xx, \yy)) = \begin{cases}
                                  \pi(\alpha_{\theta}(\xx), \beta_{\theta}(\yy)) & \xx \in U_{1} \setminus V(f_{1}), \yy \in U_{2} \setminus V(f_{2}) \\
                                  \pi(\mathbf{0}_{n}, \beta_{\theta}(\yy)) & \xx \in V(f_{1}), \yy \in U_{2} \setminus V(f_{2}) \\
                                  \pi(\alpha_{\theta}(\xx), \mathbf{0}_{m}) & \xx \in U_{1} \setminus V(f_{1}), \yy \in V(f_{2}). 
                                 \end{cases} 
\]
Note that $\{ \gamma_{\theta}\}$ is well-defined and a continuous one-parameter family of homeomorphisms such that 
$\tilde{f}(\gamma_{\theta}(\zz)) = e^{i\theta}\tilde{f}(\zz)$, where $\zz \in \pi(E)$ and 
$\theta \in [0, 2\pi]$. 
Hence $\{ \gamma_{\theta}\}$ gives the local triviality of $\tilde{f}$. 
Then the monodromy of $\tilde{f}$ can be identified with $\alpha_{2\pi}*\beta_{2\pi}$ up to homotopy. 
Here the map $\alpha_{2\pi}*\beta_{2\pi}$ is defined by 
\[
\alpha_{2\pi}*\beta_{2\pi}([\xx, s, \yy]) = [\alpha_{2\pi}(\xx), s, \beta_{2\pi}(\yy)], 
\]
where $[\xx, s, \yy] \in X_{\mathbf{t}} * Y_{\mathbf{t}}$.

By Lemma $7$, the fiber of $\tilde{f}$ is homotopy equivalent to the fiber of $f$. 
Since $D^{2}_{\rho} \setminus \{ \mathbf{0}_{2}\}$ is a CW-complex and 
$\tilde{f}^{-1}(\mathbf{t})$ is homotopy equivalent to $f^{-1}(\mathbf{t})$ for 
any $\mathbf{t} \in D^{2}_{\rho} \setminus \{ \mathbf{0}_{2}\}$, 
$\tilde{f}$ is fiber homotopy equivalent to $f$ \cite{D}. 
Then the monodromy of the tubular Milnor fibration of $f$ is equal to $\alpha_{2\pi}*\beta_{2\pi}$. 
\end{proof}

\begin{proof}[Proof of Corollary $1$]
By Theorem $2$ and the condition (iii), the fiber of the spherical Milnor fibration of $f$ 
is homotopy equivalent to 
$X_{\mathbf{t}} * Y_{\mathbf{t}}$, where $0 < \| \mathbf{t} \| \ll 1 $. 
By the condition (iii), $X_{\mathbf{t}}$ and $Y_{\mathbf{t}}$ are diffeomorphic to 
the fibers of the spherical Milnor fibrations of $f_1$ and $f_2$ respectively. 
This completes the proof. 
\end{proof}

Let $F_j$ be the fiber of the spherical Milnor fibration of $f_j$ 
which satisfies the assumptions in Section $2.2$ for $j=1, 2$. 
By \cite{M}, the reduced homology $\tilde{H}_{n+m-1}(F_{1} * F_{2})$ satisfies 
\[
\tilde{H}_{n+m-1}(F_{1} * F_{2}) = \sum_{i+j = n+m-2} \tilde{H}_{i}(F_{1}, \Bbb{Z}) \otimes \tilde{H}_{j}(F_{2}, \Bbb{Z}) + 
\sum_{i'+j' = n+m-3}\text{Tor}(\tilde{H}_{i'}(F_{1}, \Bbb{Z}), \tilde{H}_{j'}(F_{2}, \Bbb{Z})).
\]
Let $F$ be the fiber of the spherical Milnor fibration of $f = f_{1} + f_{2}$ and 
$\tau : F \rightarrow F_{1} * F_{2}$ be the homotopy equivalence in Theorem $2$. 
Then $f$ also satisfies the assumptions in Section $2.2$ and 
we have the following commutative diagram: 
\def\mapright#1{\smash{\mathop{\longrightarrow}\limits^{{#1}}}}
\def\mapdown#1{\Big\downarrow\rlap{$\vcenter{\hbox{$#1$}}$}}
\[
\begin{matrix}
\tilde{H}_{n+m-1}(F, \Bbb{Z})&\mapright{\gamma_{*}}& \tilde{H}_{n+m-1}(F, \Bbb{Z}) \\
\mapdown{\tau}&&\mapdown{\tau}\\
\tilde{H}_{n-1}(F_{1}, \Bbb{Z}) \otimes \tilde{H}_{m-1}(F_{2}, \Bbb{Z}) & \mapright{\alpha_{*} \otimes \beta_{*}} & \tilde{H}_{n-1}(F_{1}, \Bbb{Z}) \otimes \tilde{H}_{m-1}(F_{2}, \Bbb{Z}) 
\end{matrix},
\]
where $\alpha_{*}, \beta_{*}$ and $\gamma_{*}$ are the linear transformations induced by the monodromy of 
the spherical Milnor fibrations of $f_{1}, f_{2}$ and $f$ respectively. 
Since the eigenvalues of the linear transformation 
$\alpha_{*} \otimes \beta_{*} : \tilde{H}_{n-1}(F_{1}, \Bbb{Z}) \otimes \tilde{H}_{m-1}(F_{2}, \Bbb{Z}) 
\rightarrow \tilde{H}_{n-1}(F_{1}, \Bbb{Z}) \otimes \tilde{H}_{m-1}(F_{2}, \Bbb{Z})$ are given by the product of 
the eigenvalues of $\alpha_{*}$ and $\beta_{*}$, 
we obtain the following corollary.

\begin{corollary}
Assume that $f_{1}$ and $f_{2}$ satisfy the assumptions in Section $2.2$.
Let $\tilde{\zeta}_{1}(t), \tilde{\zeta}_{2}(t)$ and $\tilde{\zeta}(t)$ of the reduced zeta functions defined by 
$\alpha_{*}, \beta_{*}$ and $\gamma_{*}$ respectively. Then the divisors of the reduced zeta functions 
are related by 
\[
(\tilde{\zeta}(t)) = (\tilde{\zeta}_{1}(t))\cdot(\tilde{\zeta}_{2}(t)). 
\]
\end{corollary}

\begin{Example}
Let $f_{1}(z_{1}, z_{2})$ and $f_{2}(w_{1}, w_{2})$ be mixed polynomials of independent variables. 
In \cite[Corollary 7.6]{B}, we can choose $f_1$ and $f_2$ such that  
$K_{f_{j}}$ is the figure-$8$ knot and 
$f_j$ has the spherical Milnor fibration for $j = 1, 2$.
Then $\tilde{\zeta}_{1}(t)$ and $\tilde{\zeta}_{2}(t)$ are equal to $t^{2} -3t +1$. 
Note that $\tilde{\zeta}_{1}(t)$ and $\tilde{\zeta}_{2}(t)$ are not cyclotomic polynomials. 
By Corollary $3$, we have 
\[
(\tilde{\zeta}(t)) = 2\langle 1\rangle + \Bigl\langle \frac{7 + 3\sqrt{5}}{2}\Bigr\rangle + \Bigl\langle \frac{7 - 3\sqrt{5}}{2}\Bigr\rangle. 
\] 
\end{Example}

\section{Seifert forms of simple links defined by mixed functions}
Let $K$ be a \textit{link} in the $(2k+1)$-sphere $S^{2k+1}$, 
i.e., $K$ is an oriented codimension-two closed smooth submanifold in $S^{2k+1}$. 
A link $K$ is said to be \textit{fibered} if 
there exists a trivialization $K \times D^{2} \rightarrow N(K)$ of 
a tubular neighborhood $N(K)$ of $K$ in $S^{2k+1}$ 
and a fibration of the link exterior $E(K) = S^{2k+1} \setminus$ Int$(N(K))$, 
$\xi_{1} : E(K) \rightarrow S^1$ 
such that $\xi_{0}|\partial N(K) = \xi_{1}|\partial N(K)$, 
where $\xi_{0} : N(K) \rightarrow D^2$ is a trivialization $K \times D^{2} \rightarrow N(K)$ 
composed with the second factor. 
This fibration is also called an \textit{open book decomposition} of $S^{2k+1}$. 
A fiber of $\xi_{1}$ is called a \textit{fiber surface of the fibration of $K$}. 
If $f(\zz, \bar{\zz})$ is convenient strongly non-degenerate, $K_f$ is a fibered link \cite{O1}.

We assume that a fibered link $K$ in $S^{2k+1}$ is $(k - 2)$-connected and 
its fiber surface $F$ is $(k - 1)$-connected. 
Then $K$ is called a \textit{simple fibered link}.
Let $\alpha, \beta \in \tilde{H}_{k}(F; \Bbb{Z})$ and 
$a$ and $b$ be cycles on $F$ 
representing $\alpha$ and $\beta$ respectively. 
Set 
\[
L_{K}(\alpha, \beta) := \text{link}(a^{+}, b), 
\]
where $a^{+}$ is a pushed off of $a$ to the positive side of $F$ by a transverse vector field 
and $\text{link}(a^{+}, b)$ is the linking number of $a^{+}$ and $b$. 
The \textit{Seifert form} $L_{K}$ of $K$ is the non-singular bilinear form 
\[
L_{K} : \tilde{H}_{k}(F; \Bbb{Z}) \times \tilde{H}_{k}(F; \Bbb{Z}) \rightarrow \Bbb{Z}  
\]
on the $k$-th homology group $\tilde{H}_{k}(F; \Bbb{Z})$ 
with respect to a choice of basis of $\Tilde{H}_{k}(F; \Bbb{Z})$. 
By \cite{KN}, we can show the following proposition. 

\begin{proposition}[\cite{KN}]
Let $f_{1} : (\Bbb{C}^{n}, O_{n}) \rightarrow (\Bbb{C}, 0)$ and  
$f_{2} : (\Bbb{C}^{m}, O_{m}) \rightarrow (\Bbb{C}, 0)$ be mixed function germs of 
independent variables 
which satisfy the condition (iii). 
Suppose that 
the origin is an 
isolated singularity of $f_j$ and 
$K_{f_{j}}$ is a simple fibered link for $j = 1, 2$. 
Then 
$L_{K_{f}}$ is congruent to $(-1)^{nm}L_{K_{f_{1}}} \otimes L_{K_{f_{2}}}$. 
\end{proposition}

Kauffman and Neumann studied Seifert forms of non-simple fibered links. 
See \cite{KN}.

Let $A = (a_{i,j})$ and $A'$ be integral unimodular matrices. We say that $A'$ is an \textit{extension of $A$} 
if $A'$ is congruent to 

\[     \left(
        \begin{array}{@{\,}ccc|c@{\,}}
        a_{1,1} & \ldots & a_{1,n} & 0 \\ 
        \vdots & \vdots & \vdots & \vdots \\
        a_{n,1} & \ldots & a_{n,n} & 0 \\ \hline
        b_1    & \ldots & b_n    & \varepsilon 
        \end{array}
       \right),  \]
where 
$n$ is the rank of $A$, $b_i \in \Bbb{Z}, i= 1, \dots, n$ and 
$\varepsilon = \pm 1$. 
Let $K$ and $K'$ be simple fibered links in  $S^{2k+1}$. 
Set $F$ and $F'$ to be the fiber surfaces of $K$ and $K'$ respectively. 
If a fiber surface $F$ is obtained from $F'$ by a plumbing of a Hopf band, 
the Seifert form of $F$ is an extension of the Seifert form of $F'$ (cf. \cite{L}). 
If $k \geq 3$, the fiber surface is a positive Hopf band 
(resp. a negative Hopf band) if and only if 
its Seifert form is $(+1)$ (resp. $(-1)$). 
If a fiber surface is obtained from a disk by successive plumbings of Hopf bands 
then its Seifert form becomes a unimodular lower triangular matrix 
for a suitable choice of the basis. 
D. Lines studied high dimensional fibered knots by using plumbings \cite{L}. 

\begin{lemma}[\cite{L}]
Let $F$ and $F'$ be the fiber surfaces of simple fibered links $K$ and $K'$ in $S^{2k+1}$, 
where $k \geq 3$. 
Then $F'$ is obtained from $F$ by plumbing a Hopf band if and only if 
$L_{K'}$ is an extension of $L_{K}$. 
\end{lemma}

\begin{Example}
Suppose that $\alpha_{j} \neq \alpha_{j'} \ (j \neq j')$ and $m \geq 2$. 
Then we define a mixed polynomial and a complex polynomial as follows: 
\[
 f_{1}(\zz) =  (z_{1} + \alpha_{1}z_{2})(z_{1} + \alpha_{2}z_{2}) 
\overline{(z_{1} + \alpha_{3}z_{2})}, \ \ 
f_{2}(\ww) = \textstyle\sum_{j=1}^{m} w_{j}^{2}. 
\]
By \cite{In2}, the Seifert form of $K_{f_{1}}$ is equal to 
\[
\begin{pmatrix}
  0 & -1 \\   
  -1 & 2
\end{pmatrix}.
\]
Since the Seifert form of $K_{f_{2}}$ is equal to $(-1)^{\frac{m(m-1)}{2}}$ \cite[Proposition 2.2]{Du}, 
by Proposition $1$, the Seifert form of $K_f$ is equal to that of $K_{f_{1}}$, where 
$f = f_{1} + f_{2}$. 
Then the Seifert form of $K_f$ satisfies 
\begin{equation*}
\begin{split}
&(-1)^{\frac{m(m-1)}{2}}\begin{pmatrix}
  0 & -1 \\   
  -1 & 2
\end{pmatrix}
\rightarrow (-1)^{\frac{m(m-1)}{2}}\begin{pmatrix}
         0 & 1\\
         1  & 2 
         \end{pmatrix} \\
    \rightarrow  &(-1)^{\frac{m(m-1)}{2}}\begin{pmatrix}
         0 & 1 & 0\\
         1 & 2 & 0\\
         0 & 0 & 1 
         \end{pmatrix}
 \rightarrow  (-1)^{\frac{m(m-1)}{2}}\begin{pmatrix}
         1 & 0 & 0\\
         0 & 1 & 0\\
         0 & 0 & -1 
         \end{pmatrix}.
\end{split}
\end{equation*}
See the proof of \cite[Lemma 6]{L}. 
The links $K_{f_{1}}$ and $K_{f_{2}}$ are simple fibered links. 
By Corollary~1, $K_f$ is also a simple fibered link. 
By Lemma $8$, 
the Milnor fiber of $f$ is obtained from a disk by 
plumbing three Hopf bands 
and deplumbing a Hopf band. 
\end{Example}

By using the notion of strongly non-degenerate mixed functions and Proposition $1$, 
we show a generalization of \cite[Corollary 3]{S1}.

\begin{corollary}
Let $f_{j}(z_{j})$ be a strongly non-degenerate mixed polynomial of $1$-variable $z_j$ for $j = 1, \dots, n$. 
Set $m_{j}$ to be the mapping degree of $f_{j}/\lvert f_{j}\rvert : 
S^{1}_{\varepsilon_{j}} = \{ z_{j} \in \Bbb{C} \mid \lvert z_{j}\rvert = \varepsilon_{j} \}
\rightarrow S^{1}$ and 
\[
g_{j}(z_{j}) = \begin{cases}
           z_{j}^{m_{j} + \ell_{j}}\bar{z}_{j}^{\ell_{j}} & m_{j} > 0 \\
           z_{j}^{\ell_{j}}\bar{z}_{j}^{-m_{j} + \ell_{j}} & m_{j} < 0 
           \end{cases},
\]
where $0 < \varepsilon_{j} \ll 1$ for $j = 1, \dots, n$. 
Suppose that the Newton boundary of $f_j$ is equal to that of $g_j$ for $j = 1, \dots, n$. 
For any $j \in \{1, \dots, n\}$, assume that there exists 
an analytic family $f_{j, t}$ of strongly non-degenerate mixed polynomials 
such that $f_{j, 0} = f_{j}, f_{j, 1} = g_{j}$ and 
the Newton boundaries of $f_{j,t}$ is constant for $0 \leq t \leq 1$. 
Then the Milnor fibration of $f(\zz) = f_{1}(z_{1}) + \cdots + f_{n}(z_{n})$ is equivalent to 
that of $g(\zz) = g_{1}(z_{1}) + \cdots + g_{n}(z_{n})$.
Set the $(\lvert m\rvert -1) \times (\lvert m\rvert -1)$ matrix $\Lambda'_{m}$ as follows: 
\[
\Lambda'_{m} = \begin{cases}
            \Lambda_{m} & m > 0 \\
            {}^{t}\Lambda_{m} & m < 0 
           \end{cases},
\]
where $\Lambda_{m}$ is the $(m-1) \times (m-1)$ matrix given by 
\[
\Lambda_{m} = \left(
        \begin{array}{@{\,}cccccccc@{\,}}
1 & 0 & \ldots & \ldots & 0 \\
-1 & 1 & \ddots & \ddots & \vdots \\ 
0 & \ddots & \ddots & \ddots & \vdots \\
\vdots & \ddots & \ddots & \ddots & 0 \\
0 & \ldots & 0 & -1 & 1
\end{array}
       \right).
\]
Then we have 
\[
\Gamma_{K_{f}} \cong (-1)^{\frac{n(n+1)}{2}}\Lambda'_{m_{1}} \otimes \cdots \otimes \Lambda'_{m_{n}}. 
\]
\end{corollary}
\begin{proof}
Since $f_{j, 0}$ is strongly non-degenerate, the tubular Milnor fibration of $f_t$ is equivalent 
to that of $f_{0}$ for $0 \leq t \leq 1$ \cite[Theorem 3.14]{EO}. 
Thus the tubular Milnor fibrations of $f_j$ is equivalent to that of $g_j$ for $j = 1, \dots, n$. 
By \cite{O1} and Corollary $1$, the spherical Milnor fibration of $f$ are equivalent to that of $g$. 
By Proposition $1$ and \cite[Corollary 3]{S1}, the Seifert form $\Gamma_{K_{g}}$ is congruent to 
\[
(-1)^{\frac{n(n+1)}{2}}\Lambda'_{m_{1}} \otimes \cdots \otimes \Lambda'_{m_{n}}.
\]
This completes the proof. 
\end{proof}

\section{Enhanced Milnor numbers of simple links defined by mixed functions}
Let $K$ be a fibered link in $S^{2k+1}$. By gluing $\xi_{0}$ and $\xi_{1}$, 
we give a piecewise smooth map $\xi : S^{2k+1} \rightarrow D^{2}$. 
By \cite{KN}, $\xi$ can be extended to a continuous map $\Xi : B^{2k+2} \rightarrow D^{2}$  
which is a smooth submersion except at $\mathbf{0}_{2}$ and a corner along $\partial N(K)$. 
Then we consider the following map: 
\[
B^{2k+2} \setminus \{\mathbf{0}_{2k+2}\} \rightarrow G(2k, 2k+2), \ \ 
\xx \mapsto \ker D(\Xi(\xx)), 
\]
where $D(\Xi(\xx))$ is the differential of $\Xi$ at $\xx$ and 
$G(2k, 2k+2)$ is the Grassman manifold of oriented $2k$-planes in $\Bbb{R}^{2k+2}$. 
This map defines an element of $\pi_{2k+1}(G(2k, 2k+2))$. 
Note that $\pi_{2k+1}(G(2k, 2k+2))$ is isomorphic to  
\[
\begin{cases}
\Bbb{Z} \oplus \Bbb{Z} & k=1 \\
\Bbb{Z} \oplus \Bbb{Z}/2\Bbb{Z} & k > 1
\end{cases}.
\]
The homotopy class of $\Xi$ has the form $((-1)^{k+1}\mu(K), \lambda(K))$. 
This pair $((-1)^{k+1}\mu(K), \lambda(K))$ is called the \textit{enhanced Milnor number of $K$} and 
$\lambda(K)$ is called the \textit{enhancement to the Milnor number}. See \cite{NR1, NR2, NR3}. 
Note that if $K$ is a fibered link coming from an isolated singularity of a complex hypersurface, 
$\lambda(K)$ always vanishes. 
By \cite{NR2}, we have 

\begin{theorem}[\cite{NR2}]
Let $f_{1} : (\Bbb{C}^{n}, O_{n}) \rightarrow (\Bbb{C}, 0)$ and  
$f_{2} : (\Bbb{C}^{m}, O_{m}) \rightarrow (\Bbb{C}, 0)$ be mixed function germs of independent variables. 
Assume that $f_{1}, f_{2}$ and $f$ satisfy the condition (1). 
Suppose that $O_n$ and $O_m$ are isolated singularities of $f_1$ and $f_2$.  
Then $\mu(K_{f}) = \mu(K_{f_{1}})\mu(K_{f_{2}})$ and 
$\lambda(K_{f}) \equiv \lambda(K_{f_{1}})\mu(K_{f_{2}}) + \mu(K_{f_{1}})\lambda(K_{f_{2}}) \bmod 2$. 
\end{theorem}

For any $\ell \in \Bbb{N}$ and $k \geq 2$, 
there exists a $(k+1)$-variables Brieskorn polynomial $P$ such that 
$((-1)^{k+1}\mu(K_{P}), \lambda(K_P)) = ((-1)^{k+1}\ell, 0)$. See \cite{M2}. 
By Theorem $4$ and \cite{In1}, we calculate the enhanced Milnor numbers of simple fibered links 
defined by mixed polynomials of join type. Then we have 

\begin{theorem}
For any $\ell \in \Bbb{N}$, 
there exists a $(k+1)$-variables mixed polynomial $P = P_{1} + P_{2}$ of join type such that
$P_{1}, P_{2}$ and $P$ satisfies the condition (1) and 
$K_P$ is a simple fibered link which satisfies 
$((-1)^{k+1}\mu(K_{P}), \lambda(K_P)) = ((-1)^{k+1}\ell, 1)$, where $k \geq 2$. 
\end{theorem}

\begin{proof}
We define a mixed polynomial and a complex polynomial as follows: 
\[
 f_{1}(\zz) =  (z_{1}^{p} + \alpha_{1}z_{2})(z_{1}^{p} + \alpha_{2}z_{2}) 
\overline{(z_{1}^{p} + \alpha_{3}z_{2})}, \ \ 
f_{2}(\zz) = z_{1}^{2} + \bar{z}_{2}^{2}, \ \
f_{3}(\ww) = \textstyle\sum_{i=1}^{m} w_{i}^{a_{i}}, 
\]
where $\alpha_{j} \neq \alpha_{j'} \ (j \neq j'), a_{j} \geq 2$ and $m \geq 1$. 
Then $f_j$ is a convenient strongly non-degenerate mixed polynomial
and $K_{f_{j}}$ is a simple fibered link for $j = 1, 2, 3$. 
By \cite{EN, In1, M2}, we have 
\begin{equation*}
\begin{split}
(\mu(K_{f_{1}}), \lambda(K_{f_{1}})) &= (2p, 1), \ \
(\mu(K_{f_{2}}), \lambda(K_{f_{2}})) = (1, 1), \\
(\mu(K_{f_{3}}), \lambda(K_{f_{3}})) &= ((a_{1} - 1)\cdots(a_{m} - 1), 0). 
\end{split}
\end{equation*}
If $\ell$ is a positive even integer, 
we set $p = \frac{\ell}{2}, P_{1} = f_{1}$ and $P_{2} = f_{3}$. 
By Corollary $1$, $K_P$ is also a simple fibered link. 
By Theorem $4$, we have 
\[
(\mu(K_{P}), \lambda(K_{P})) = (\ell(a_{1} - 1)\cdots(a_{m} - 1), (a_{1} - 1)\cdots(a_{m} - 1)\bmod 2). 
\]
We set $a_{i} = 2$ for $i = 1, \dots, m$. 
Then $((-1)^{k+1}\mu(K_{P}), \lambda(K_P))$ is equal to $((-1)^{k+1}\ell, 1)$. 

If $\ell$ is a positive odd integer, put $P_{1} = f_{2}$ and $P_{2} = f_{3}$. Then we have 
\[
(\mu(K_{P}), \lambda(K_{P})) = ((a_{1} - 1)\cdots(a_{m} - 1), (a_{1} - 1)\cdots(a_{m} - 1)\bmod 2). 
\]
We set 
$a_{1} = \ell +1$ and $a_{i} = 2$ for $i = 2, \dots, m$. 
Then $((-1)^{k+1}\mu(K_{P}), \lambda(K_P))$ is equal to $((-1)^{k+1}\ell, 1)$. 
\end{proof}

\end{document}